\documentclass[a4paper]{amsart}

\usepackage[english]{babel}
\usepackage[T1]{fontenc}
\usepackage{amsmath,amssymb}
\usepackage{amsfonts}
\usepackage{amsthm}
\usepackage{graphicx,epsfig}
\usepackage{comment}

\theoremstyle{plain}
\newtheorem{theorem}{Theorem}
\newtheorem{lemma}[theorem]{Lemma}
\newtheorem{proposition}[theorem]{Proposition}

\theoremstyle{definition}
\newtheorem{definition}{Definition}[section]

\theoremstyle{remark}
\newtheorem*{remark}{Remark}

\numberwithin{equation}{section}

\newcommand\dr{\mathbb}
\renewcommand\d{\textnormal}

\newcommand\st{\textrm{ such that }}

\newcommand\G{\Gamma}
\newcommand\g{\gamma}

\newcommand\GG{{G / \G}}
\newcommand\GH{{H  \backslash G}}
\renewcommand\phi{\varphi}

\title[]{A brief remark on orbits of $\d{SL}(2,\dr Z)$ in the euclidean plane}

\author{Antonin Guilloux}

\address{UMPA - ENS Lyon\\46 All\'ee d'Italie\\FR - 69007 LYON\\
http://www.umpa.ens-lyon.fr/$\sim$aguillou/\\ antonin.guilloux@umpa.ens-lyon.fr}

\begin{document}

\maketitle

F. Ledrappier \cite{ledrappier} proved the following theorem as an application of Ratner theorem on unipotent flows (A. Nogueira \cite{nogueira} proved it for $\d{SL}(2,\dr Z)$ with different methods):
\begin{theorem}[Ledrappier, Nogueira]\label{the:led}
Let $\G$ be a lattice of $\d{SL}(2,\dr R)$ of covolume $c(\G)$, $\|.\|$ the euclidean norm on the algebra of $2\times 2$-matrices $\mathcal M(2,\dr R)$, and $v\in \dr R^2$ with non-discrete orbit under $\G$.

Then we have the following limit, for all $\phi\in \mathcal C_c(\dr R^2\setminus\{0\})$:
$$\frac{1}{T}\sum_{\g\in \G,\; \|\g\|\leq T} \phi(\g v)\xrightarrow{T\to \infty} \frac{1}{|v|c(\G)}\int_{\dr R^2\setminus\{0\}} \phi(w) \frac{dw}{|w|} \; .$$
\end{theorem}

We may draw a picture of this equidistribution theorem, for example with $\G=SL(2,\dr Z)$. Here is shown the orbit of the point $\begin{pmatrix} 1\\ \frac{\pi}{2}\end{pmatrix}$ under the ball of radius $1000$. We draw only the points which are falling in some compact to avoid the rescaling of the picture (in the theorem, $\phi$ has to have compact support):

\begin{figure}[ht]
\begin{center}
\includegraphics[scale=.5]{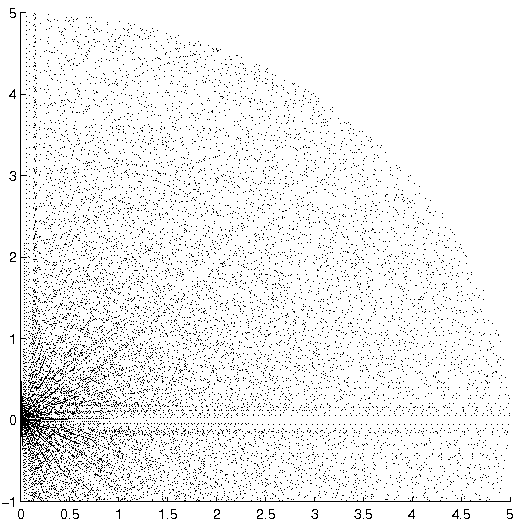} \qquad
\end{center}
\caption{Orbit under the ball of radius 1000}
\end{figure} 

A striking phenomenon is the gaps around lines of simple rational slopes. This appears for any initial point. We will describe here these gaps in a fully elementary way for the lattice $\G=\d{SL}(2,\dr Z)$ (which is enough to describe it for all arithmetic lattices). Let us mention that our analysis is carried on in the arithmetic case for sake of elementariness but a similar analysis can be done for non-arithmetic lattices.

Another experimentation with a cocompact lattice does not show these gaps. It will be clear from the analysis below that this comes from the unique ergodicity of the unipotent flow in $\d{SL}(2,\dr Z)\backslash \d{SL}(2,\dr R)$.

\section{The plane and the horocycles}

The key point in the theorem of Ledrappier is the identification of $\dr R^2\setminus \{0\}$ and the space of horocycles $SL(2,\dr R)/U$, where $U=\{u(t)=\begin{pmatrix} 1 & t\\ 0& 1\end{pmatrix} \d{ for } t \in \dr R\}$ is the upper triangular unipotent subgroup of $\d{SL}(2,\dr R)$. The projection from $\d{SL}(2,\dr R)$ to the plane is given by the first column of the matrix.
We will use the following section from $\dr R^2\setminus (\{0\}\times \dr R)$ to $\d{SL}(2,\dr R)$:
$$\sigma \begin{pmatrix} a\\b\end{pmatrix} \mapsto \begin{pmatrix} a & 0\\b & a^{-1}\end{pmatrix}\; .$$
Then we have: $\sigma \begin{pmatrix} a\\b\end{pmatrix} u(t)= \begin{pmatrix} a & ta\\b & a^{-1}+tb\end{pmatrix}$, which in turn projects to the same point $\begin{pmatrix} a\\ b\end{pmatrix}$.

The theorem of Ledrappier is proven using the fact that a large portion of  a dense orbit of $U$ in $\d{SL}(2,\dr Z)\backslash \d{SL}(2,\dr R)$ becomes equidistributed in this space. Without any more detail on this, we may just remark that if $\frac{a}{b}$ is rational, the orbit $\sigma\begin{pmatrix} a\\b\end{pmatrix}U$ projects in a periodic horocycle in $\d{SL}(2,\dr Z)\backslash \d{SL}(2,\dr R)$. This means that the application $\dr R \to \d{SL}(2,\dr Z)\backslash \d{SL}(2,\dr R)$ given by $t\mapsto \d{SL}(2,\dr Z)\sigma\begin{pmatrix} a\\b\end{pmatrix} u(t)$ is periodic. Another way to state it: there exists $t\in \dr R^*$ and $\gamma \in \d{SL}(2,\dr Z)$ such that we have $\g \sigma\begin{pmatrix} a\\b\end{pmatrix}=\sigma\begin{pmatrix} a\\b\end{pmatrix}u(t)$. The period of this application is called the period of  the orbit $\sigma\begin{pmatrix} a\\b\end{pmatrix}U$.

\section{Periods and heights of points with rational slope}

Consider a point $v_0=\begin{pmatrix} a\\b\end{pmatrix}$ in $\dr R^2\setminus\{0\}$ with $\frac{b}{a}\in \dr Q$ or $a=0$. Then we may define the following number:
\begin{definition}
\item The \emph{period} $\rho(v_0)$ of $v_0$ is the period of the orbit $\sigma(v_0)U$ in the space $\d{SL}(2,\dr Z)\backslash \d{SL}(2,\dr R)$
\end{definition}

It is not hard to effectively compute this period:
\begin{proposition}
Write $v_0=t\begin{pmatrix} p\\q\end{pmatrix}$ with $p$ and $q$ two coprime integers.

Then the period of $v_0$ is given by $\rho(v_0)=\frac{1}{t^2}$.
\end{proposition}
\begin{proof}
We assume here that $p\neq 0$ (if not you can change the section $\sigma$). The point $t\begin{pmatrix} p\\q\end{pmatrix}$ correspond via $\sigma$ to the matrix $\begin{pmatrix} tp& 0\\tq&(tp)^{-1}\end{pmatrix}$. So we have to solve the equation: $\gamma\begin{pmatrix} tp& 0\\tq&(tp)^{-1}\end{pmatrix}=\begin{pmatrix} tp& 0\\tq&(tp)^{-1}\end{pmatrix}u(s)$ for $\gamma$ in $\d{SL}(2,\dr Z)$ and $s$ real. That is: $$\begin{pmatrix} t(ap+bq)& b(tp)^{-1}\\t(cp+dq)&d(tp)^{-1}\end{pmatrix}=\begin{pmatrix} tp& stp\\tq&stq+(tp)^{-1}\end{pmatrix}\;,$$ for $a$, $b$, $c$, $d$ integers verifying $ad-bc=1$ and $s$ real. We check that $b$ and $d-1$ have to be divisible by $p$ hence $s$ has to belong to $\frac{1}{t^2}\dr Z$. Now we easily check the following equality, thus proving the proposition :
$$\begin{pmatrix}1+pq& p^2\\q^2& 1-pq\end{pmatrix}\begin{pmatrix} tp& 0\\tq&(tp)^{-1}\end{pmatrix}=\begin{pmatrix} tp& 0\\tq&(tp)^{-1}\end{pmatrix}\begin{pmatrix} 1& \frac{1}{t^2}\\0&1\end{pmatrix}\;.$$
\end{proof}
This computation is an elementary way to check that the period of a point with rational slope is invariant under the action of $\d{SL}(2,\dr Z)$: the image under an element of $\d{SL}(2,\dr Z)$ of a point $\begin{pmatrix} p\\ q\end{pmatrix}$ with coprime $p$ and $q$ is still a point of this form. Of course, a more intrinsic way to see this is to look at the definition of the period which is clearly invariant under $\d{SL}(2,\dr Z)$. Anyway this simple fact is the key remark. Indeed the set of points of fixed period is a discrete subset of the plane. Call $\mathcal P(\rho):=\{v \in \dr R^2 \d{ of rational slope with period }\rho\}$. The previous proposition describe these sets as $\mathcal P(\rho)=\frac{1}{\sqrt{\rho}}\dr Z\wedge\dr Z$ where $\dr Z\wedge \dr Z$ stands for the set of points with coprime integer coordinates.

Moreover we may define the height of a point of rational slope (using the height function on the space $\dr P_1(\dr Q)$) by this simple formula: $h(t\begin{pmatrix} p\\q\end{pmatrix})=\sqrt{p^2+q^2}=|\begin{pmatrix} p\\q\end{pmatrix}|$ (as usual $p$ and $q$ are coprime integers).
We have the following tautological formula for any point $v$ of rational slope in the plane : $$\rho(v)|v|^2=h(v)^2\; .$$

\section{Spectrum of periods}

Consider $v$ a point in the plane (not $0$). Then for each $\rho>0$, the distance of $v$ to the set $\mathcal P(\rho)$ is a nonnegative real number. Moreover if $v$ has irrational slope, this number is positive for each $\rho$. We then define a function, called spectrum of periods, for $v$:
\begin{definition}
Let $v$ be a point in the plane of irrational slope. Then its \emph{spectrum of periods} $\mathcal D_v$ is the function :$$\mathcal D_v : \begin{matrix} \dr R_+^*&\to &\dr R_+^*\\ \rho & \mapsto & d(v,\mathcal P(\rho))\end{matrix}$$
\end{definition}
The description of the sets $\mathcal P(\rho)$ made above allows the following rewriting of $\mathcal D_v$: $\mathcal D_v(\rho)=\frac{1}{\sqrt{\rho}}d(\sqrt{\rho} v,\dr Z \wedge \dr Z)$. This last expression shows that for $\rho$ big enough this function encodes the diophantine property of the slope of $v$, and may be interesting to study precisely. But a first remark is that $\mathcal D_v(\rho)$ is always smaller than $\frac{1}{\sqrt{\rho}}$; moreover for $\rho \leq \frac{1}{(2|v|)^2}$, $\mathcal D_v(\rho)$ is bigger than $\frac{1}{2\sqrt{\rho}}$:
\begin{lemma}
For $\rho \leq \frac{1}{(2|v|)^2}$, we have $\frac{1}{2\sqrt{\rho}}\leq \mathcal D_v(\rho)\leq \frac{1}{\sqrt{\rho}}$. Moreover, as $\rho \to 0$, $\mathcal D_v(\rho)$ is equivalent to $\frac{1}{\sqrt{\rho}}$.
\end{lemma}
\begin{proof}
If $\rho$ is less than $\frac{1}{(2|v|)^2}$, the modulus of $\sqrt{\rho} v$ is less than $\frac{1}{2}$. So its distance to $\dr Z\wedge \dr Z$ is more than $\frac{1}{2}$, proving the inequality. The equivalence is straightforward.
\end{proof}

\bigskip

We are now able to state the desired property: the orbit of $v$ under the set $\Gamma_T=\{\gamma \in \d{SL}(2,\dr Z) \st \|\gamma\|\leq T\}$ cannot come too close of the points of rational slopes. 
\begin{proposition}\label{main}
Let $w$ be a point of rational slope in the plane. Then the distance of $\Gamma_T v$ to $w$ is bounded from below by $\frac{\mathcal D_v(\rho(w))}{T}=\frac{\mathcal D_v(\frac{h(w)^2}{|w|^2})}{T}$.
\end{proposition}

Let us prove the proposition before giving a more geometric description.
\begin{proof}
Consider an element $\gamma$ of $\d{SL}(2,\dr Z)$ of euclidean norm less than $T$. Then it multiplies length by at most $T$ Let us suppose that the point $\gamma v$ is very close to some $w$ with rational slope: $|\gamma v - w|= \frac{\epsilon}{T}$ for some $\epsilon$ ; we immediately get that $|v- \gamma^{-1}w|\leq \epsilon$. But the point $\gamma^{-1}w$ has same period as $w$ by invariance and thus belongs to $\mathcal P(\rho(w))$. So by definition of $\mathcal D_v$ and the tautological formula on the period, we get that $\gamma v$ cannot be too close to $w$: $$|\gamma v - w|\geq \frac{\mathcal D_v(\rho(w))}{T}\geq \frac{\mathcal D_v(\frac{h(w)^2}{|w|^2})}{T}\; .$$
\end{proof}

Now if we are interested at how the orbit of $v$ comes close some half-line of rational slopes $\dr R_+^*\begin{pmatrix} p\\ q \end{pmatrix}$, we fix the height $h(w)$. If we furthermore add the condition $|w|\geq {2|v|h(w)}$ we may use the easy bound on $\mathcal D_v$ to get: $$|\gamma v - w|\geq  \frac{|w|}{2h(w)T}\; , \d{for all }\gamma\d{ of norm less than } T.$$
We see on this last formula that the simpler is the slope (as a rational number) the harder it is to come close. The linear behavior suggests a picture in coordinates (radius, slope) to see clearly the gaps. Here we draw the whole orbit (check that the radius  of points goes up to 1900) for $T=1000$ in a small neighborhood of the horizontal axis. The gap is fairly evident.  The graphs of the functions $\frac{\mathcal D_v}{1000}$ and $-\frac{\mathcal D_v}{1000}$ are drawn in blue. The previous proposition states that no point of this orbit may fall between this two graphs. Once again we are in coordinates (radius,slope):

\begin{figure}[ht]
\begin{center}
\includegraphics[scale=.45]{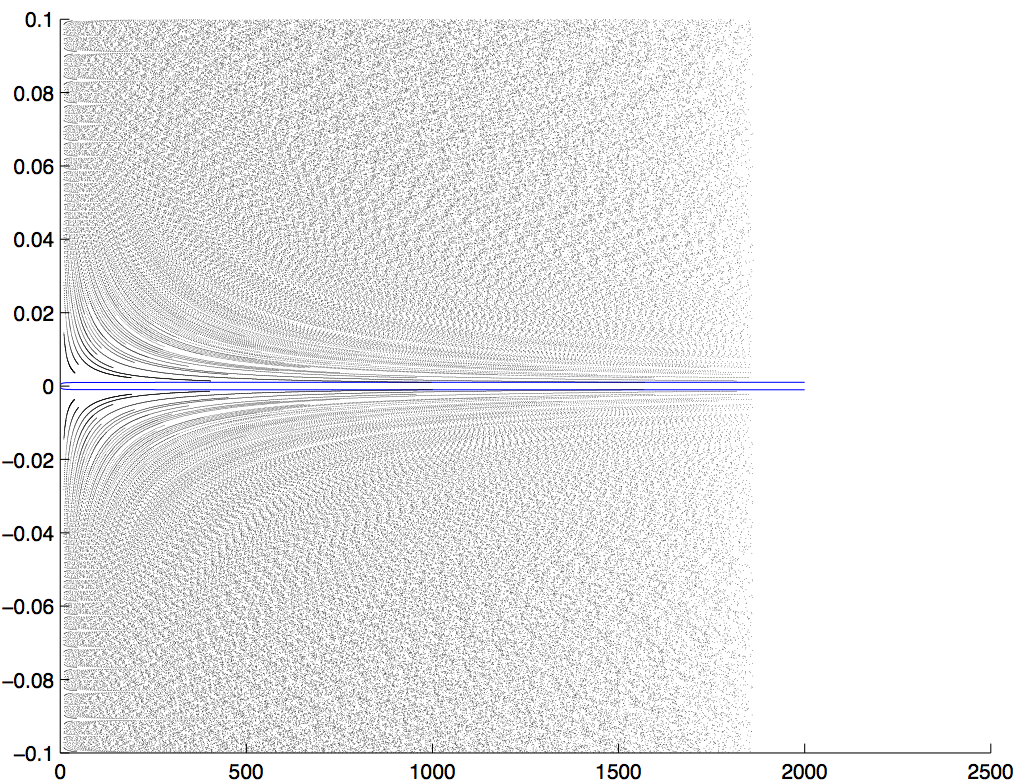} \qquad
\end{center}
\caption{The gap around the horizontal axis}
\end{figure}

\section{Two kinds of optimality}

Let us mention that the optimality of the described gap seen on the previous picture is easy to understand. Indeed the next lemma states that some points of the orbit $\Gamma_T v$ are almost as close as possible to points of rational slope.
\begin{lemma}
There exist a $\gamma \in \d{SL}(2,\dr Z)$ with $\|\gamma\|\leq T$ and some point $w$ of rational slope such that we have for all $T\geq 10$: $$ |\gamma v -w| -\frac{\mathcal D_v(\rho(w))}{T} \leq 10\frac{\mathcal D_v(\rho(w))}{T^2}\; .$$
\end{lemma}
\begin{proof}
Consider the matrix $\gamma=\begin{pmatrix} 1& T-1\\0&1\end{pmatrix}$ of $\Gamma_T$. Let us note $v = \begin{pmatrix} a\\b \end{pmatrix}$. Let us assume first that $|a|\leq b$. Then we have $\gamma v =\begin{pmatrix} a+(T-1)b\\b \end{pmatrix}$. Now consider the point $w =\begin{pmatrix} a+(T-1)b\\0 \end{pmatrix}$ of slope $0$. First we get that the distance $|\gamma v -w|$ is equal to $|b|$. Second we check that $$(T-4)|b|\leq \mathcal D_v(\rho(w) \leq T |b|$$ using the formula for the function $\mathcal D_v$.
That means that we have:
\begin{eqnarray*}
\frac{\mathcal D_v(\rho(w))}{T}&\leq& |b|-4\frac{|b|}{T}\\
&\leq & |\gamma v -w| - 4\frac{T}{T-4}\frac{\mathcal D_v(\rho(w))}{T^2}\\
&\leq & |\gamma v -w| - 10\frac{\mathcal D_v(\rho(w))}{T^2}
\end{eqnarray*}

So the lemma is proven in this case. If we had $|b|<|a|$, we may then consider the matrix $\gamma=\begin{pmatrix} 1& 0\\T-1&1\end{pmatrix}$ and the point $w=\begin{pmatrix} 0\\ (T-1)a+b\end{pmatrix}$ which lead to the same estimate via the same computation !
\end{proof}

\bigskip 

But of course this consideration is somehow deceptive, as it describes a general fact verified for any initial point and do not reflects the diophantine properties of this point. So let us show that the diophantine information about the beginning point effectively lies in the evolution of the orbit. Consider a point $v=\begin{pmatrix} a \\ b \end{pmatrix}$ with irrational slope $s=\frac{b}{a}$. Recall that the best approximation of $s$ by a rational number $\frac{p}{q}$ gives us the point $\begin{pmatrix} a \\ a\frac{p}{q} \end{pmatrix}$ which realizes the distance $\mathcal D_v(\frac{q^2}{a^2})$. Hence we are only interested in the periods $\rho$ of the form $\frac{q^2}{a^2}$.

According to the following lemma,  we do always get points in the orbit under a ball of big enough size $T$ which almost realizes the minimal predicted distance $\frac{\mathcal D_v(\rho)}{T}$  to the set $\mathcal P(\rho)$.
  
\begin{lemma}
Let $v=\begin{pmatrix} a \\ b \end{pmatrix}$ be a point with irrational slope, and fix $\varepsilon>0$.
Then there exists a real $T_0$ such that for all $T>T_0$, and every integer $q>1$, there is a point $\gamma.v$ in $\G_T.v$ and a point $w$ in $\mathcal P(\rho')$ such that:
 $$|\frac{\rho'}{\frac{q^2}{a^2}}-1|\leq \frac{2(1+\varepsilon)D_v(\frac{q^2}{a^2})}{aT}$$ and the distance between $\gamma.v$ and $w$ is at most $$(1+\varepsilon)\frac{\mathcal D_v(\frac{q^2}{|a|^2})}{T}\; .$$
\end{lemma}

\begin{proof}
Once again the proof is elementary. We just have to find in $\G_T$ a contracting element $\gamma$ and apply it to a well-chosen  vector. I let the reader verify that the following construction verifies the above estimates. Take $N$ the biggest integer such that $N^2+2\leq T^2$, and consider the matrix $\gamma=\left(\begin {matrix} N & -1\\ 1&0\end{matrix}\right)$ of $\G_T$. This matrix contracts the vector $\begin{pmatrix} 1 \\ N \end{pmatrix}$ to the vector $\begin{pmatrix} 0 \\ 1 \end{pmatrix}$.

Hence, let $w_0=\begin{pmatrix} a \\ b' \end{pmatrix}$ be the point of $\mathcal P(\frac{q^2}{a^2})$ realizing the infimum distance $\mathcal D_v(\frac{q^2}{a^2})$. Eventually, consider $\alpha$ and $\lambda$ the solutions of $$\alpha w_0 - v = \lambda \begin{pmatrix} 1 \\ N \end{pmatrix}$$ (which has solutions for all but possibly one integer $N$).We have $\alpha=\frac{Na-b}{Na-b'}$ and $\lambda=\frac{a(b-b')}{Na-b'}$.

Now consider $w=\gamma(\alpha w_0)$. We have: $$w-\gamma.v=\gamma (\alpha w_0 - v)=\lambda \begin{pmatrix} 0 \\ 1 \end{pmatrix}$$
Hence the distance between $w$ and $\gamma.v$ is $\lambda$ which is as near as wanted of $\frac{\mathcal D_v(\frac{q^2}{|a|^2})}{T}$ (recall that $\mathcal D_v(\frac{q^2}{a^2})=b-b'$).

Moreover the period $\rho'$ of $w$ is the one of $\alpha w_0$, i.e. $\frac{\rho}{\alpha^2}$. Hence we get the desired control on $\rho'$ by checking that, for $N$ big enough (but independent of $q$):

$$|\frac{1}{\alpha^2}-1|=\frac{(Na-b')^2-(Na-b)^2}{(Na-b')^2}\leq \frac{(1+\varepsilon)2D_v(\frac{q^2}{a^2})}{Ta}$$
\end{proof}

\smallskip

This previous result allows us to get the best rationnal approximation of the slope  by the following limit:

\begin{proposition}
Let $q$ be a positive integer and $v=\begin{pmatrix} a \\ b \end{pmatrix}$ be a point with irrational slope $s=\frac{b}{a}$. 

Then we have the following equality:

$$\inf\left\{|s-\frac{p}{q}| \d{ for } p\in \mathbb Z\right\} = \lim_{T\to \infty} \frac{T}{a} \inf\left\{d(\G_T.v, \mathcal P(\rho'))\d{ for } \left|\frac{a^2\rho'}{q^2}\right|\leq \frac{2}{qT}\right\}$$
\end{proposition}

\begin{proof}
The previous lemma ensure that the limsup of the right side is correct.

So we just have to prove that the liminf is bigger than the left-hand side: let $\rho'$ belong to the segment $[\frac{q^2}{a^2}-\frac{2q}{Ta^2} ; \frac{q^2}{a^2}+\frac{2q}{Ta^2}]$, $w$ be a point in $\mathcal P(\rho')$ and $\gamma \in \G_T$ be such as $d(\gamma .v,w)\leq \frac{aA}{T}$.

Then, as usual, we get $D_v(\rho')\leq d(v,\gamma^{-1}w)\leq aA$. And, as the formulas given for $D_v$ show, $D_v(\rho')-D_v(\rho) =O(|\frac{\rho'}{\rho}-1|)$. We conclude by seeing that $|\frac{\rho'}{\rho}-1|$ is a big $O$ of $\frac{1}{T}$. Hence, we have $aA\geq \mathcal D_v(\rho) +O(\frac{1}{T})$, which proves that the liminf is greater than $\frac{D_v(\rho)}{a}=\inf\left\{|s-\frac{p}{q}| \d{ for } p\in \mathbb Z\right\}$.
\end{proof}

\begin{remark}
Of course this is not a valid way to compute the left-hand side of the equality ! It only shows that we may find the dipophantine information in the orbit, hence gives us the hope that one may find a direct proof of some results on diophantine approximation from this viewpoint and generalize it to other situations (see below).
\end{remark}

\bigskip

Eventually let's restrict our attention to some compact, for example an annulus $A$. Ledrappier's theorem describe the asymptotic distribution of the sets $\Gamma_T v\cap A$, i.e.  the points of the orbit of $v$ under $\Gamma_T$ which are inside $A$. Around every line $L$ of rational slope and for every positive $T$, the proposition \ref{main} gives us a domain of area (in fact the cone over a Cantor set) $\frac{c_L}{T}$ - where $c(L)$ only depends on $L$- in which no point of $\Gamma_T v \cap A$ lies. So, globally speaking, we have found a set of area at least $\frac{c}{T}$, for some constant $c$, such that no point of the orbit of $\Gamma_T v \cap A$ falls in this set.

As Ledrappier's theorem implies that the number of points in $\Gamma_T v\cap A$ is equivalent to a constant times $T$, the information given by proposition \ref{main} seems to be a valuable one.

\section{Generalizations}

This concluding section is a mostly speculative one and far less elementary than the previous description. The point is that the method and the result concerning the repartition of the orbits of $\d{SL}(2,\dr Z)$ in the plane has been generalized, for example by Gorodnik \cite{gorodnik}, Gorodnik-Weiss \cite{goroweiss} Ledrappier-Pollicott \cite{ledrappier-pollicott} and the author \cite{poldyn-latorb} to a wide variety of situations, which may be described with some simplifications as follows.

Let $G$ be a closed simple subgroup of $\d{GL}(n,\dr R)$ or $\d{GL}(n,\dr Q_p)$ or a finite product of them. Let $H$ be a closed subgroup of $G$ that is either unipotent or simple (or semidirect product of them, but with additional assumptions \cite{poldyn-latorb}), and $\G$ a lattice in $G$. As $G$ is included in a matrix algebra, we may choose a norm to compute the size of an element of $\G$ thus defining the ball $\G_T$. Remark that in all these known cases, any lattice of $H$ is finitely generated.
Let $x$ be a point of $\GH$ with dense orbit under $\G$. Then the repartition of the orbit $\G_T.x$ in $\GH$ may be described in the same way as in theorem \ref{the:led}.\footnote{I do not want to state it precisely, nor will I be very precise in the following, as the settings require some technical hypotheses useless to discuss here.}

For example, orbits of $\d{SL}(n,\dr Z)$ in $\dr R^n$ belong to the known situations. And the same analysis as before leads to exactly the same conclusions, including the diophantine part. Moreover, we may give a description of the gaps in a more general situation. Suppose that $\GH$ is embedded in a vector space, on which $G$ acts linearily and the $G$-actions are compatible. Then $\GH$ may be equipped with a distance coming from a norm on the vector space. This situation is not so rare and may be found under some hypotheses using Chevalley's theorem \cite{borel}. Moreover suppose $H$ has closed orbit in $\GG$.

We check below that the set of points in $\GH$ corresponding to closed orbit of $H$ in $\GG$ of a given covolume $\rho$ is a closed set. If this holds, the distance from a given point $x$ of dense orbit to this set is defined and strictly positive, and the ball $\G_T$, as a finite set of invertible linear transformations, has a bounded contraction. Hence we follow the description of the gaps made before for $\d{SL}(2,\dr Z)$ without difficulties.

So we conclude this paper on the following (may be well-known) lemma:
\begin{lemma}
Let $G$ be a locally compact group, $H$ a closed subgroup of $G$  with all its lattices finitely generated and $\G$ a lattice in $G$ such that $H\cap \G$ is a lattice in $H$ of covolume one (to normalize the Haar measure on $H$). Suppose that, if $g^n$ belongs to $H$ for some $g\in G$ and $n$ integer, then $g$ belongs to $H$.

Then, for all $\rho>0$, the subset $\mathcal P(\rho)$ of $\GH$ consisting of classes $Hg$ such that $g\G g^{-1}\cap H$ is a lattice in $H$ of covolume $\rho$ is a closed set.
\end{lemma} 

\begin{remark}
I tried to state it in a general enough setting, so there is in the statement the two ad-hoc hypotheses I need below. It is easy to check that in the above described cases they are fulfilled. 
\end{remark}

\begin{proof}
Let $x_n=H g_n$ be a sequence of points in $\mathcal P(\rho)$ converging to $x=Hg$ in $\GG$. Suppose we made the choices such that $g_n$ converges to $g$ in $G$.

Let $A$ be a compact subset in $H$ of volume strictly greater than $\rho$. Then, by definition, for every $n$, there is an element $\gamma_n$ in $\G$ such that $A \cap g_n\gamma_n g_n^{-1}A$ is not empty. As $A$ is compact and $g_n$ tends to $g$, the choices for the $\gamma_n$'s stay inside a compact subset, hence are in finite number. So there is a fixed $\gamma\in \G$ such that for infinitely many $n$, the intersection $A \cap g_n\gamma g_n^{-1}A$ is not empty. Conclusion: $A \cap g\gamma g^{-1}A$ is not empty and $g\G g^{-1}\cap H$ is a lattice in $H$ of covolume at most $\rho$.

We now prove that $g\G g^{-1}\cap H$ effectively has covolume $\rho$ in $H$. We even prove the  stronger fact: the sequence of subgroups $\G \cap g_n^{-1} H g_n$ is a stationnary sequence. Hence for $n$ and $m$ big enough, $g_n^{-1} g_m$ normalizes $H$ and let its Haar measure invariant. The subgroup of the normalizer of $H$ letting its Haar measure invariant is closed, so $g_n^{-1} g$ belongs to it, thus proving that $g\G g^{-1} \cap H$ is of covolume $\rho$.

As $g\G g^{-1}\cap H$ is finitely generated, we just have to show that for any $\gamma \in \G$, if $g^{-1} \gamma g$ is in $H$, then $g_n^{-1} \gamma g_n$ is in $H$ for $n$ big enough. So let $A$ be a compact subset in $H$ of positive measure $\alpha$ such its images under $g^{-1}\G g\cap H$ are disjoints. And let $A'$ be of the form $A'=\cup_{i=0}^k g^{-1}\gamma gA$, where $k$ is bigger than $\frac{\rho}{\alpha}$. Then for all $n$, there exist a $\gamma_n\in \G\cap g_n^{-1}Hg_n$ such that, $g_n \gamma_n g_n^{-1}A'\cap A'$ is not empty. As before, there is only a finite number of possibilities for $\gamma_n$, hence it takes some value $\gamma'$ infinitely many times. Therefore $g \gamma' g^{-1} A' \cap A'$ is not empty. By construction, $\gamma'$ is a power $\gamma^k$ of $\gamma$, and for infinitely many $n$, $g_n\gamma^k g_n^{-1}$ belongs to $H$. Now the hypothesis on $H$ shows that $g_n\gamma g_n^{-1}$ also belongs to $H$.

At this point we showed that for any $\gamma$ in $\G\cap g^{-1}H g$, there is an infinite number of $n$ such that $\gamma$ belongs to $\G\cap g_n^{-1}H g_n$. Using this fact along any subsequence, it shows that for $n$ big enough, $\gamma$ belongs to $\G\cap g_n^{-1}H g_n$. And for $n$ big enough, $\G\cap g_n^{-1}H g_n$ contains all the generators of $\G\cap g^{-1}H g$. Hence for $n$ big enough, the subgroups $\G\cap g_n^{-1}H g_n$ and $\G\cap g^{-1}H g$ are the same one. This concludes the proof of this lemma.
\end{proof}

\bibliographystyle{amsplain}
\bibliography{diophuni}
\end{document}